\documentclass[11pt,leqno]{article}
\usepackage{amsmath,amsfonts,amscd,amsthm}
\begin{document}

\newcommand{\se}{\setcounter{equation}{0}}
\def\theequation{\thesection.\arabic{equation}}

\newtheorem{theorem}{Theorem}[section]
\newtheorem{cdf}{Corollary}[section]
\newtheorem{lemma}{Lemma}[section]
\newtheorem{remark}{Remark}[section]

\newcommand{\e}{{\bf e}}
\newcommand{\f}{{\bf f}}

\newcommand{\bu}{{\bf u}}
\newcommand{\bv}{{\bf v}}
\newcommand{\bw}{{\bf w}}
\newcommand{\by}{{\bf y}}
\newcommand{\bz}{{\bf z}}
\newcommand{\bH}{{\bf H}}
\newcommand{\bJ}{{\bf J}}
\newcommand{\bL}{{\bf L}}

\newcommand{\bchi}{\mbox{\boldmath $\chi$}}
\newcommand{\bta}{\mbox{\boldmath $\eta$}}
\newcommand{\bphi}{\mbox{\boldmath $\phi$}}
\newcommand{\bth}{\mbox{\boldmath $\theta$}}
\newcommand{\bxi}{\mbox{\boldmath $\xi$}}
\newcommand{\bzt}{\mbox{\boldmath $\zeta$}}

\newcommand{\hf}{\hat{\f}}
\newcommand{\hta}{\hat{\bta}}
\newcommand{\hth}{\hat{\bth}}
\newcommand{\hbu}{\hat{\bu}}
\newcommand{\hxi}{\hat{\bxi}}
\newcommand{\hby}{\hat{\by}}
\newcommand{\hzt}{\hat{\bzt}}

\newcommand{\td}{\tilde{\Delta}}
\newcommand{\bub}{\bu_{\beta}}
\newcommand{\bubt}{\bu_{\beta,t}}

\newcommand{\bbu}{\bar{\bu}}
\newcommand{\bby}{\bar{\by}}
\newcommand{\bbz}{\bar{\bz}}

\title
{ \large\bf Nonlinear Galerkin Finite Element for Viscoelastic Fluid Flow: \\
Optimal Error Estimate}

\author{Deepjyoti Goswami}
\date{}
\maketitle

\begin{abstract}
In this article, we discuss a couple of nonlinear Galerkin method (NLG) in finite
element set up for viscoelastic fluid flow, mainly equations of motion arising in
the flow of $2D$ Oldroyd model. We obtain improved error estimate in $L^{\infty}(\bL^2)$ norm, which is optimal in nature, for linear finite element approximation, in view of the error estimate available in literature, in $L^2(\bH^1)$ norm.
\end{abstract}

\vspace{0.30cm} 

\vspace{.2in}
\section{Introduction}
Nonlinear Galerkin (NLG) method has been introduced by Marion and Temam in \cite{MT89} as a means of turbulence modeling, in order to handle long-time integration of dissipative evolution partial differential equations. In a classical Galerkin method (CGM), the equation is projected to some finite dimensional space, thereby neglecting the orthogonal space. 
But it is well-known that nonlinear dynamics is sensitive to initial data; a small change
in the initial data, after a large time, may result in a significant change in the
system. Hence, for large time integration, it is only appropriate to capture the effect of
some the neglected terms. The idea behind NLG is to take into account the effect of the neglected terms in the long run.

The method in \cite{MT89} is based on the eigenvectors of the underlying linear elliptic operator. Later in \cite{MT90}, it is expanded to more general bases and more specifically to finite elements. Unlike Galerkin finite element method, where we work with a finite element space $\bJ_h$, with grid size $h$, in nonlinear Galerkin finite element method, we work
with two finite element spaces, namely $\bJ_H$ on a coarser grid and $\bJ_h^H$, a suitable
complement of $\bJ_H$ in $\bJ_h$, respectively. It corresponds to spatial splitting of the unknown $\bu\in \bJ_h$ as
$$ \bu=\by+\bz\in \bJ_H+\bJ_h^H. $$
In other words, we work with two set of equations in nonlinear Galerkin, one in $\by$
and the other in $\bz$. But the equation in $\bz$, which is solved in a finer grid, is generally simplified (linearized and time derivative is discarded). And the full equation in $\by$ is solved only in a coarser grid. In a way, nonlinear Galerkin method, in the context of finite element, results in a two-level scheme.

In this article, we consider semi-discrete nonlinear Galerkin approximations to the 
following system of equations of motion arising in the Oldroyd model (see (\cite{GP11})) of order one:
\begin{eqnarray}\label{om}
 ~~\frac {\partial \bu}{\partial t}+\bu\cdot\nabla\bu-\mu\Delta\bu-\int_0^t \beta
 (t-\tau)\Delta\bu (x,\tau)\,d\tau+\nabla p=\f(x,t),~~x\in \Omega,~t>0
\end{eqnarray}
with incompressibility condition
\begin{eqnarray}\label{ic}
 \nabla \cdot \bu=0,~~~x\in \Omega,~t>0,
\end{eqnarray}
and initial and boundary conditions
\begin{eqnarray}\label{ibc}
 \bu(x,0)=\bu_0~~\mbox {in}~\Omega,~~~\bu=0,~~\mbox {on}~\partial\Omega,~t\ge 0.
\end{eqnarray}
Here, $\Omega$ is a bounded domain in $\mathbb{R}^2$ with boundary $\partial \Omega$,
$\mu = 2 \kappa\lambda^{-1}>0$ and the kernel $\beta (t) = \gamma \exp (-\delta t),$
where $\gamma= 2\lambda^{-1}(\nu-\kappa \lambda^{- 1})>0$ and $\delta =\lambda^{-1}>0$.
For further details, we would like to refer to \cite{GP11} and references therein.

We are interested in optimal error estimate of the nonlinear Galerkin approximations of the
velocity of (\ref{om})-(\ref{ibc}). Note that the results obtained here for the above problem are also valid for Navier-Stokes' equations. Error estimations of nonlinear Galerkin
methods in mixed finite element set up for Navier-Stokes' equations are carried out in \cite{AM94}. Both nonlinearity and time dependence are treated on coarse space. It is
proved that
\begin{align*}
\|(\bu^h-\bu_h)(t)\|_{\bH^1(\Omega)} \le c(t)H^2, \\
\|(p^h-p_h)(t)\|_{\bL^2(\Omega)} \le c(t)H^2,
\end{align*}
where $(\bu_h,p_h)$ is the classical Galerkin approximation and $(\bu^h,p^h)$ is the nonlinear Galerkin approximation. These results are improved to $O(H^3)$ in \cite{MX95},
and  also similar result is obtained for $L^{\infty}(\bL^2)$-norm error estimate for velocity approximation, but only for semi-linear parabolic problems. The author feels that it is not straight forward to carry forward these results to Navier-Stokes equations or to the equations being considered here. The proof of \cite{MX95} depends on one important estimate involving the nonlinear term $f(\bu)$, namely
$$|(f(\by_h+\bz_h)-f(\by_h),\bchi)| \le \|\bz_h\|_{\bL^2(\Omega)}\|\bchi\|_{\bL^2(\Omega)}.$$
(For the notation, kindly refer to \cite{MX95})
But it may not be possible to establish similar estimate for the nonlinear term $f(\bu)=(\bu\cdot\nabla)\bu$. On the other hand, it is observed that the improvement in the order of convergence is due to the fact that the splitting in space is done on the basis of $\bL^2$ projections, unlike previous approaches, where the splitting is based on hierarchical basis. Similar approach is adopted in this article.

No further improvement is observed by the author for (piecewise) linear finite element discretization and with forcing term $\f\in\bL^2(\Omega)$. Similar results are observed in \cite{HL96,HL98}, but for $\f\in\bH^1(\Omega)$ and for Navier-Stokes equations. In \cite{NR97}, various nonlinear Galerkin finite elements are studied in depth for one-dimensional problems and similar results are obtained, although termed as optimal in nature. For example, the method is studied for piecewise polynomials of degree $2n-1$. And the error estimates obtained in energy-norm and $L^2$-norm are as follows
\begin{align*}
\|(\bu_h-\bu^h)(t)\|_{\bH^n(\Omega)} &\le c(t)H^{3n-1}, \\
\|(\bu_h-\bu^h)(t)\|_{\bL^2(\Omega)} &\le c(t)\min\{t^{-1/2}H^{4n-1},H^{7n/2-1}\},
\end{align*}
respectively. For $n=1$, we have the piecewise linear finite element approximation and
error estimates are of $O(H^2)$ and $O(H^{5/2})$, in energy-norm and $L^2$-norm,
respectively. Later on He {\it et. al} have studied NLG and modified NLG in both spectral and finite element set ups, \cite{CEHL99,HL99,HLL99,HWCL03,HMMC04} to name a few. In \cite{CEHL99}, (modified) spectral nonlinear Galerkin method is applied for the problem (\ref{om})-(\ref{ibc}) with periodic boundary condition. And in \cite{HLL99,HWCL03,HMMC04},
convergence and stability is analyzed for the fully discrete NLG.

Recently in \cite{LH10}, a new projection is employed for a two-level finite element for Navier-Stokes equations. Error estimates of $O((\log h)^{1/2}H^3)$ and $O((\log h)^{1/2} H^4)$, in energy-norm and $L^2$-norm, respectively, are obtained. The finite dimensional case of the Brezis-Gallouet inequality plays a crucial role apart from the projection in the error analysis and also the reason for the appearance of the logarithmic term. Note that here the forcing term is taken in $L^{\infty}(0,T;(\bL^2(\Omega))^2)\cap L^2(0,T;\bH_0^1(\Omega))$.
The projection in \cite{LH10} is based on fully discretization; finite element for space discretization and backward Euler for time discretization. However, Crank-Nicolson scheme is used for the problem.

It is the observation of the author that although the energy-norm error estimate for NLG is of order $O(H^3)$, there is no $L^2$-norm error estimate of order $O(H^4)$. As is mentioned in \cite{MX95}, it is an open question and the author tries to resolve this for a couple of (modified) NLG methods in this article. Although, this improved error estimate is proved for
the problem (\ref{om})-(\ref{ibc}), it will go through for Navier-Stokes equations.

The author would like to remark here that two NLG methods are considered here only for theoretical purpose. In the first NLG, nonlinearity is preserved in both the equations and as a result, is not physically viable due to very high cost of implementation. But by comparing the error analyses for both the methods, we can identify the key factors for the gain in order of convergence. In both the cases nonlinearity (in error) is separated and treated only after the linearized error is estimated. It is observed that the nonlinear terms play an important role, as they can limit the order of convergence, see Remark \ref{rmk}.

It is to be noted that the approach here is an heuristic one and the author does not claim that NLG performs better than classical Galerkin. The arguments in this regard are settled
to certain extent in the articles \cite{HR93,NR97,GP08}, to name a few. 


The article is organized as follows. In section $2$, we briefly review the Oldroyd model.
Section $3$ deals with the classical Galerkin approximation. In Section $4$, we present the nonlinear Galerkin modelsand finally in Section $5$, we discuss the error analysis.

\section{Preliminaries}
\se

For our subsequent use, we denote by bold face letters the $\mathbb{R}^2$-valued
function space such as
\begin{align*}
 \bH_0^1 = [H_0^1(\Omega)]^2, ~~~ \bL^2 = [L^2(\Omega)]^2.
\end{align*}
Note that $\bH^1_0$ is equipped with a norm
$$ \|\nabla\bv\|= \big(\sum_{i,j=1}^{2}(\partial_j v_i, \partial_j
 v_i)\big)^{1/2}=\big(\sum_{i=1}^{2}(\nabla v_i, \nabla v_i)\big)^{1/2}. $$
Further, we introduce divergence free function spaces:
\begin{align*}
\bJ_1 &= \{\bphi\in\bH_0^1 : \nabla \cdot \bphi = 0\} \\
\bJ= \{\bphi \in\bL^2 :\nabla \cdot \bphi &= 0~~\mbox{in}~~
 \Omega,\bphi\cdot{\bf n}|_{\partial\Omega}=0~~\mbox {holds weakly}\}, 
\end{align*}
where ${\bf n}$ is the outward normal to the boundary $\partial \Omega$ and $\bphi
\cdot {\bf n} |_{\partial\Omega} = 0$ should be understood in the sense of trace
in $\bH^{-1/2}(\partial\Omega)$, see \cite{temam}.
For any Banach space $X$, let $L^p(0, T; X)$ denote the space of measurable $X$
-valued functions $\bphi$ on  $ (0,T) $ such that
$$ \int_0^T \|\bphi (t)\|^p_X~dt <\infty~~~\mbox {if}~~1 \le p < \infty, $$
and for $p=\infty$
$$ {\displaystyle{ess \sup_{0<t<T}}} \|\bphi (t)\|_X <\infty~~~\mbox {if}~~p=\infty. $$
Through out this paper, we make the following assumptions:\\
(${\bf A1}$). For ${\bf g} \in \bL^2$, let the unique pair of solutions $\{\bv
\in\bJ_1, q \in L^2 /R\} $ for the steady state Stokes problem
\begin{align*}
 -\Delta\bv + \nabla q = {\bf g}, ~~
 \nabla \cdot\bv = 0~~\mbox {in}~\Omega,~~~\bv|_{\partial\Omega}=0,
\end{align*}
satisfy the following regularity result
$$  \| \bv \|_2 + \|q\|_{H^1/R} \le C\|{\bf g}\|. $$
\noindent
(${\bf A2}$). The initial velocity $\bu_0$ and the external force $\f$ satisfy for
positive constant $M_0,$ and for $T$ with $0<T \leq \infty$
$$ \bu_0\in\bJ_1,~\f,\f_t \in L^{\infty} (0, T ;\bL^2)~~~\mbox{with} ~~~
\|\bu_0\|_1 \le M_0,~~{\displaystyle{\sup_{0<t<T} }}\big\{\|\f\|,
 \|\f_t\|\big\} \le M_0. $$

\noindent Before going into details, let us introduce weak formulation of
(\ref{om})-(\ref{ibc}).
Find  $\bu(t) \in {\bf J}_1,~t>0 $ such that
\begin{equation}\label{wfj}
 (\bu_t, \bphi) +\mu (\nabla \bu, \nabla \bphi )+( \bu \cdot \nabla \bu, \bphi)
 +\int_0^t \beta(t-s) (\nabla\bu(s),\nabla \bphi)ds=(\f,\bphi),~\forall\bphi
 \in {\bf J}_1.
\end{equation}
We would like to note the positivity property of the integral. For a proof, we would
like to refer to \cite{GP11,G11}.
\begin{lemma}\label{positivity}
For arbitrary $\alpha>0$, $t^*>0$ and $\phi\in L^2(0,t^*)$, the following positive
definite property holds
$$ \int_0^{t^*}{\left(\int_0^t{\exp{[-\alpha(t- s)]}\phi(s)}\,ds
   \right)\phi(t)}\,dt\ge 0. $$
\end{lemma}
\noindent
We will use this non-negativity property frequently in this article without exclusive
mention of the Lemma. Further, for existence and uniqueness and the regularity of the solution of the problem (\ref{wfj}), we refer to \cite{GP11,G11} and references cited therein.

\section{Classical Galerkin Method}
\se

From now on, we denote $h$ with $0<h<1$ to be a real positive discretization
parameter tending to zero. Let  $\bH_h$ and $L_h$, $0<h<1$ be two family of
finite dimensional subspaces of $\bH_0^1 $ and $L^2$, respectively,
approximating velocity vector and the pressure. Assume that the following
approximation properties are satisfied for the spaces $\bH_h$ and $L_h$: \\
${\bf (B1)}$ For each $\bw \in\bH_0^1 \cap \bH^2 $ and $ q \in
H^1/R$ there exist approximations $i_h w \in \bH_h $ and $ j_h q \in
L_h $ such that
$$ \|\bw-i_h\bw\|+ h \| \nabla (\bw-i_h \bw)\| \le K_0 h^2 \| \bw\|_2,
 ~~~~\| q - j_h q \| \le K_0 h \| q\|_1. $$
Further, suppose that the following inverse hypothesis holds for $\bw_h\in\bH_h$:
\begin{align}\label{inv.hypo}
 \|\nabla \bw_h\| \leq  K_0 h^{-1} \|\bw_h\|.
\end{align}
For defining the Galerkin approximations, set for $\bv, \bw, \bphi \in \bH_0^1$,
$$ a(\bv, \bphi) = (\nabla \bv, \nabla \bphi) $$
and
$$ b(\bv, \bw,\bphi)= \frac{1}{2} (\bv \cdot \nabla \bw , \bphi)
   - \frac{1}{2} (\bv \cdot \nabla \bphi, \bw). $$
Note that the operator $b(\cdot, \cdot, \cdot)$ preserves the antisymmetric property of
the original nonlinear term, that is,
$$ b(\bv_h, \bw_h, \bw_h) = 0 \;\;\; \forall \bv_h, \bw_h \in {\bH}_h. $$

\noindent In order to consider a discrete space, analogous to $\bJ_1$, we
impose the discrete incompressibility condition on $\bH_h$ and call it as
$\bJ_h$. Thus, we define $\bJ_h,$ as
$$ {\bf J}_h = \{ v_h \in {\bf H}_h : (\chi_h,\nabla\cdot v_h)=0
 ~~~\forall \chi_h \in L_h \}. $$
Note that $\bJ_h$ is not a subspace of $\bJ_1$. With $\bJ_h$ as above, we now introduce
the following weak formulation as: Find $\bu_h(t)\in {\bf J}_h $ such that $\bu_h(0) =
\bu_{0h} $ and for $\bphi_h \in \bJ_h,t>0$
\begin{equation}\label{dwfj}
~~~~ (\bu_{ht},\bphi_h) +\mu a (\bu_h,\bphi_h)+ \int_0^t \beta(t-s) a(\bu_h(s), \bphi_h)~ds
= -b( \bu_h, \bu_h, \bphi_h)+(\f,\bphi_h).
\end{equation}
Since $\bJ_h$ is finite dimensional, the problem (\ref{dwfj}) leads to a system of
nonlinear integro-differential equations. For global existence of a unique solution
of (\ref{dwfj}), we refer to \cite{GP11,G11}.

\noindent Moreover, we also assume that the following approximation property holds
true for ${\bf J}_h $. \\
\noindent
${\bf (B2)}$ For every $\bw \in {\bf J}_1 \cap {\bf H}^2, $ there exists an
approximation $r_h \bw \in {\bf J_h}$ such that
$$ \|\bw-r_h\bw\|+h \| \nabla (\bw - r_h \bw) \| \le K_5 h^2 \|\bw\|_2 . $$
The $L^2$ projection $P_h:\bL^2\mapsto \bJ_h$ satisfies the following properties
(see \cite{GP11}): for $\bphi\in \bJ_h$,
\begin{equation}\label{ph1}
 \|\bphi- P_h \bphi\|+ h \|\nabla P_h \bphi\| \leq C h\|\nabla \bphi\|,
\end{equation}
and for $\bphi \in \bJ_1 \cap \bH^2,$
\begin{equation}\label{ph2}
 \|\bphi-P_h\bphi\|+h\|\nabla(\bphi-P_h \bphi)\|\le C h^2\|\td\bphi\|.
\end{equation}
We now define the discrete  operator $\Delta_h: \bH_h \mapsto \bH_h$ through the
bilinear form $a (\cdot, \cdot)$ as
\begin{eqnarray}\label{do}
 a(\bv_h, \bphi_h) = (-\Delta_h\bv_h, \bphi)~~~~\forall \bv_h, \bphi_h\in\bH_h.
\end{eqnarray}
Set the discrete analogue of the Stokes operator $\td =P(-\Delta) $ as
$\td_h = P_h(-\Delta_h) $. Using Sobolev imbedding and Sobolev inequality, it is
easy to prove the following Lemma
\begin{lemma}\label{nonlin}
Suppose conditions (${\bf A1}$), (${\bf B1}$) and (${\bf  B2}$) are satisfied. Then there 
exists a positive constant $K$ such that for $\bv,\bw,\bphi\in\bH_h$, the following holds:
\begin{equation}\label{nonlin1}
 |(\bv\cdot\nabla\bw,\bphi)| \le K \left\{
\begin{array}{l}
 \|\bv\|^{1/2}\|\nabla\bv\|^{1/2}\|\nabla\bw\|^{1/2}\|\Delta_h\bw\|^{1/2}
 \|\bphi\|, \\
 \|\bv\|^{1/2}\|\Delta_h\bv\|^{1/2}\|\nabla\bw\|\|\bphi\|, \\
 \|\bv\|^{1/2}\|\nabla\bv\|^{1/2}\|\nabla\bw\|\|\bphi\|^{1/2}
 \|\nabla\bphi\|^{1/2}, \\
 \|\bv\|\|\nabla\bw\|\|\bphi\|^{1/2}\|\Delta_h\bphi\|^{1/2}, \\
 \|\bv\|\|\nabla\bw\|^{1/2}\|\Delta_h\bw\|^{1/2}\|\bphi\|^{1/2}
 \|\nabla\bphi\|^{1/2}
\end{array}\right.
\end{equation}
\end{lemma}

\noindent For examples of subspaces $\bH_h$ and $L_h$ satisfying assumptions (${\bf B1}$),
(${\bf B2}'$), and (${\bf B2}$), we again would like to refer to \cite{GP11} and references cited therein.

\noindent Note that the semi-discrete solutions admit {\it a priori} and regularity
estimates analogous to that of the continuous solution. Below, we present a Lemma
containing certain estimates of the same and once again, we would like to refer to
\cite{GP11} for a proof.
\begin{lemma}\label{est.uh}
Let $0<\alpha <\min\{\mu\lambda_1,\delta\}$, where $\lambda_1>0$ is the smallest eigenvalue
of the Stokes' operator. Let the assumptions (${\bf A1}$),(${\bf A2}$),(${\bf B1}$) and (${\bf  B2}$) hold. Then the semi-discrete Galerkin approximation $\bu_h$ of the velocity $\bu$ satisfies, for $t>0,$
\begin{eqnarray}
 \|\bu_h(t)\|+e^{-2\alpha t}\int_0^t e^{2\alpha s}\|\nabla\bu_h(t)\|^2~ds \le K,
 \label{est.uh1} \\
 \|\nabla\bu_h(t)\|+e^{-2\alpha t}\int_0^t e^{2\alpha s}\|\td_h\bu_h(s)\|^2~ds \le K,
 \label{est.uh2} \\
 (\tau^*(t))^{1/2}\|\td_h\bu_h(t)\| \le K, \label{est.uh3}
\end{eqnarray}
where $\tau^*(t)=\min\{t,1\}$. The positive constant $K$ depends only on the given data.
In particular, $K$ is independent of $h$ and $t$.
\end{lemma}

\noindent The following semi-discrete error estimates are proved in \cite{GP11}.
\begin{theorem}\label{errest}
Let $\Omega$ be a convex polygon and let the conditions (${\bf A1}$)-(${\bf A2}$)
and (${\bf B1}$)-(${\bf B2}$) be satisfied. Further, let the discrete initial velocity $\bu_{0h}\in \bJ_h$ with $\bu_{0h}=P_h\bu_0,$ where $\bu_0\in \bJ_1.$ Then,
there exists a positive constant $C$, that depends only on the given data and the
domain $\Omega$, such that for $0<T<\infty $ with $t\in (0,T]$
$$ \|(\bu-\bu_h)(t)\|+h\|\nabla(\bu-\bu_h)(t)\|\le Ce^{Ct}h^2t^{-1/2}. $$
\end{theorem}

\section{Nonlinear Galerkin Method}
\se

In this section, we work with another space discretizing parameter $H$ such that
$0<h<H$ and both $h,H$ tend to $0$. We introduce two more spaces as follows:
\begin{eqnarray}\label{space.split}
\bJ_h=\bJ_H + \bJ_h^H, ~~\mbox{with}~~\bJ_h^H= (I-P_H)\bJ_h
\end{eqnarray}
Note that, by definition, the spaces $\bJ_H$ and $\bJ_h^H$ are orthogonal with
respect to the $\bL^2$-inner product $(\cdot,\cdot)$. In practice, $\bJ_H$ corresponds
to a coarse grid and $\bJ_h^H$ corresponds to a fine grid.

\noindent The following properties are crucial for our error estimates. For a proof,
we refer to \cite{AM94}.
\begin{eqnarray}\label{jhH1}
\|\bchi\| \le cH\|\bchi\|_1,~~\bchi\in\bJ_h^H.
\end{eqnarray}
And there exists $0<\rho<1$ independent of $h$ and $H$ such that
\begin{equation}\label{jhH2}
|a(\bphi,\bchi)| \le (1-\rho)\|\bphi\|_1\|\bchi_h\|_1,~~\bphi\in\bJ_H, \bchi\in\bJ_h^H.
\end{equation}
From (\ref{jhH2}), we can easily deduce that
\begin{equation}\label{jhH3}
\rho\{\|\bphi\|_1^2+\|\bchi\|_1^2\} \le \|\bphi+\bchi_h\|_1^2,~~\bphi\in\bJ_H, \bchi\in\bJ_h^H.
\end{equation}
\noindent
In the first modified nonlinear Galerkin method (NLG I), 
we look for a solution $\bu^h$ in $\bJ_h$ such that
$$ \bu^h=\by^H+\bz^h \in \bJ_H+\bJ_h^H. $$
It is implemented on the interval $(t_0,T],~0<t_0<T<\infty.$ On the interval $(0,t_0]$, classical Galerkin method is implemented, resulting in $(\bu_h,p_h)$.
\begin{remark}
 Since the error analysis of NLG demands higher regularity of the unknown and this means
 higher singularity at $t=0$, the idea is to avoid these kinds of singularity.
\end{remark}
\noindent
For $t>t_0$, we look for $(\by^H,\bz^h)$ satisfying
\begin{align}\label{dwfJH}
(\by_t^H, \bphi) +\mu a (\bu^h,\bphi) &+ b(\bu^h,\bu^h,\bphi)+\int_{t_0}^t
\beta(t-s) a(\bu^h(s),\bphi)~ds=(\f, \bphi), \nonumber \\
\mu a (\bu^h,\bchi) &+ b(\bu^h,\bu^h,\bchi)+ \int_{t_0}^t \beta(t-s) a(\bu^h(s),
\bchi)~ds =(\f, \bchi),
\end{align}
for $\bphi\in\bJ^H,~\bchi\in\bJ_h^H$. We set $\by^H(t_0)=P_H\bu_h(t_0).$

\noindent In the second modified nonlinear Galerkin method (NLG II), we 
look for a solution $\bu^h$ in $\bJ_h$ such that
$$ \bu^h=\by^H+\bz^h \in \bJ_H+\bJ_h^H $$
satisfying
\begin{align}\label{dwfJH1}
(\by_t^H, \bphi) +\mu a (\bu^h,\bphi) &+ b(\bu^h,\bu^h,\bphi)
+ \int_{t_0}^t \beta(t-s) a(\bu^h(s),\bphi)~ds=(\f, \bphi), \nonumber \\
\mu a (\bu^h,\bchi) &+ b(\bu^h,\by^H,\bchi)+b(\by^H,\bz^h,\bchi)+ \int_{t_0}^t \beta(t-s) a(\bu^h(s),\bchi)~ds =(\f, \bchi),
\end{align}
for $\bphi\in\bJ^H,~\bchi\in\bJ_h^H$.
\begin{remark}
Here and henceforth, subscript means the classical Galerkin method and superscript
means the nonlinear Galerkin method.
\end{remark}
\begin{remark}
Both the NLGs can be heuristically derived from (\ref{dwfj}) as follows.
We split the Galerkin approximation $\bu_h$ with the help of the $\bL^2$ projection $P_H$.
\begin{equation}\label{gss} 
\bu_h=P_H\bu_h+(I-P_H)\bu_h=\by_H+\bz_h, 
\end{equation}
And we project the system (\ref{dwfj}) on $\bJ_H,\bJ_h^H$ to obtain the coupled system:
\begin{equation}\label{dwfjH}
\left\{\begin{array}{rc}
(\by_{ht},\bphi) +& \mu a (\bu_h,\bphi)+ \int_{t_0}^t \beta(t-s) a(\bu_h(s), \bphi)~ds
= -b( \bu_h, \bu_h, \bphi)+(\f,\bphi), \\
(\bz_{ht},\bchi)+& \mu a (\bu_h,\bchi)+ \int_{t_0}^t \beta(t-s) a(\bu_h(s), \bchi)~ds
= -b( \bu_h, \bu_h, \bchi)+(\f,\bchi),
\end{array}\right.
\end{equation}
for $\bphi \in \bJ_H,~\bchi \in \bJ_h^H$.
Assuming the time derivative and higher space derivatives of $\bz^h$ are small, various
(modified) NLG methods are defined. In case, time derivative of $\bz^h$ is retained in the
equation, different time-steps can be employed for the two equations, (much) smaller time-step
for the equation involving $\by^H$, for example $k1=k, k2=kl+1,~l\in\mathbb{N}$. In other words,
$\bz^h$ remains steady as $\by^H$ evolves for $l$ times at each turn (see \cite{BM97}), which in
a way is equivalent to treating an evolution equation in $\by^H$ and a steady equation in $\bz^h$.
\end{remark}

The well-posedness of both the NLGs follows easily as in the case of Navier-Stokes equations, see \cite{MT89, AM94}. But for the sake of completeness, we present below {\it a priori} estimates of the approximate solution pair $\{\by^H,\bz^h\}$. And for the sake of brevity, we only sketch a proof (similar to the proofs in \cite{GP11}). Note that the proof will go through for both the NLGs.

\begin{lemma}\label{apr}
Under the assumptions of Lemma \ref{est.uh} and for $\by^H(t_0)=P_H\bu_h(t_0)$, the solution pair $\{\by^H,\bz^h\}$ of (\ref{dwfJH}) (or (\ref{dwfJH1})) satisfies, for $t>t_0$
\begin{equation}\label{apr1}
\|\by^H\|^2+e^{-2\alpha t}\int_{t_0}^t e^{2\alpha s}\|\nabla\bu^h(s)\|^2 ds \le K.
\end{equation}
And if $H$ is small enough to satisfy
\begin{equation}\label{H.restrict}
\mu-cL_H^2\|\by^H\|^2>0,
\end{equation}
where $L_H \sim |log~h|^{1/2}$, then the following estimate holds
\begin{equation}\label{apr2}
\|\bz_h\| \le K.
\end{equation}
The constant $K>0$ depends only on the given data. In particular, $K$ is independent of $h,H$ and $t$.
\end{lemma}

\begin{proof}
Choose $\bphi=\by^H,\bchi=\bz^h$ in (\ref{dwfJH}) (or (\ref{dwfJH1})) and add the resulting equations. Multiply by $e^{2\alpha t}$, integrate from $t_0$ to $t$. Drop the resulting double integral term due to non-negativity. Multiply by $e^{-2\alpha t}$ to obtain (\ref{apr1}).

For the second estimate, we choose $\bchi=\bz^h$ in the second equation of (\ref{dwfJH}) (or (\ref{dwfJH1})) to obtain
\begin{align}\label{apr01}
\frac{3\mu}{4} \|\nabla\bz^h\|^2 \le 2\mu\|\nabla\by^H\|^2+\frac{\gamma^2}{\mu(\delta-\alpha)}
\int_{t_0}^t e^{-2\alpha(t-s)}\|\nabla\bu^h(s)\|^2 ds-b(\by^H+\bz^h,\by^H,\bz^h).
\end{align}
Use (\ref{apr1}) to bound the second term on the right hand-side of (\ref{apr01}).
Using (\ref{inv.hypo}) and (\ref{jhH1}), we find
\begin{align*}
b(\by^H,\by^H,\bz^h) & \le 2^{1/2}\|\by^H\|^{1/2}\|\nabla\by^H\|^{3/2}\|\bz^h\|^{1/2}
\|\nabla\bz^h\|^{1/2} \le c\|\by^H\|\|\nabla\by^H\|\|\nabla\bz^h\| \\
& \le \frac{\mu}{8}\|\nabla\bz^h\|^2+c\|\by^H\|^2\|\nabla\by^H\|^2, \\
b(\bz^h,\by^H,\bz^h) & \le \|\bz^h\|\|\nabla\bz^h\|\|\by^H\|_{\infty}
\le cL_H\|\bz^h\|\|\nabla\bz^h\|\|\nabla\by^H\| \\
& \le \frac{\mu}{8}\|\nabla\bz^h\|^2+cL_H^2\|\bz^h\|^2\|\nabla\by^H\|^2.
\end{align*}
We have used the finite dimensional case of Brezis-Gallouet inequality (see \cite[(3.12)]{LH10})
$$ \|\bu_h\|_{\infty} \le cL_h\|\nabla\bu_h\|;~~L_h\sim |log~h|^{1/2}. $$
Therefore, we obtain from (\ref{apr01})
$$ \frac{\mu}{2}\|\nabla\bz^h\|^2 \le K+2\mu\|\nabla\by^H\|^2+c\|\nabla\by^H\|^2
(\|\by^H\|^2+L_H^2\|\bz^h\|^2). $$
Again with the help of (\ref{inv.hypo}) and (\ref{jhH1}), we note that
$$ \mu\|\bz^h\|^2 \le KH^2+K+cL_H^2\|\by^H\|^2\|\bz^h\|^2. $$
Under the assumption on $H$, we have the required result, that is, (\ref{apr2}).
This completes the proof.
\end{proof}

\begin{remark}
Further estimates of $\by^H$ and $\bz^h$ can be obtained following the estimates proved in \cite{GP11} for small enough $H$ to satisfy (\ref{H.restrict}). 
\end{remark}

\section{Error Estimate}
\se

In this section, we work out the error between classical Galerkin approximation and
nonlinear Galerkin approximation of velocity.

\noindent Before actually working out the error estimates, we present below the Lemma
involving the estimate of $\bz_h$. For a proof, we refer to \cite{AM94}.

\begin{lemma}\label{bzh.est}
Under the assumptions Lemma 3.2 and for the solution $\bu_h$ of (\ref{dwfj}), the following
estimates are satisfied for $\bz_h=(I-P_H)\bu_h$ and for $t>t_0$
\begin{equation}\label{bzh.est1}
\left\{\begin{array}{rcl}
\|\bz_h\|+H\|\bz_h\|_1 \le K(t)H^2, \\
\|\bz_{ht}\|+H\|\bz_{ht}\|_1 \le K(t)H^2, \\
\|\bz_{htt}\|+H\|\bz_{htt}\|_1 \le K(t)H^2.
\end{array}\right.
\end{equation}
\end{lemma}

\begin{remark}
We note that Lemma \ref{bzh.est} requires the estimates of higher order time derivatives of error due to Galerkin approximation:
$$ \|(\bu-\bu_h)^{(i)}\|+h\|(\bu-\bu_h)^{(i)}\|_1 \le K(t)h^2,~~~0\le i\le 2, $$
where $(\cdot)^{(i)}$ means $i^{th}$ time derivative. For $i=0$, the result is stated in Theorem \ref{errest}. For the remaining cases, the proof consists of differentiating the error equation in time and deriving estimates. Proofs are technical and lengthy and hence are skipped for the sake of brevity.
\end{remark}

\noindent 
In order to separate the effect of the nonlinearity in the error, we introduce
$$ \bbu (\in\bJ_h)=P_H\bbu+(I-P_H)\bbu =\bby+\bbz\in \bJ_H+\bJ_h^H $$
satisfying the following linearized system ($t>t_0$)
\begin{align}\label{dwfjHl}
(\bby_t,\bphi) +& \mu a(\bbu,\bphi)+ \int_{t_0}^t \beta(t-s) a(\bbu(s), \bphi)~ds
= -b( \bu_h, \bu_h, \bphi)+(\f,\bphi)~~ \bphi \in \bJ_H \nonumber \\
& \mu a (\bbu,\bchi)+ \int_{t_0}^t \beta(t-s) a(\bbu(s), \bchi)~ds
= -b( \bu_h, \bu_h, \bchi)+(\f,\bchi)~~ \bchi \in \bJ_h^H,
\end{align}
and $\bby(t_0)=\by_H(t_0).$ Being linear it is easy to establish the well-posedness
of the above system and the following estimates.
\begin{lemma}
Under the assumptions of Lemma 3.2, we have
\begin{align*}
\|\nabla\bbu^h\|^2+e^{-2\alpha t}\int_{t_0}^t e^{2\alpha t}\|\td_h\bbu^h\|^2ds &\le K \\
\|\td_h\bbu^h\| &\le K,
\end{align*}
where the constant depends on $\bu_0$ and $\f$.
\end{lemma}

\subsection{NLG I}

We define
$$ \e :=\bu_h-\bu^h=(\by_H-\by^H)+(\bz_h-\bz^h) =: \e_1+\e_2. $$
We further split the errors as follows:
\begin{align} \label{err.split}
\left\{\begin{array}{rcl}
\e_1 &=& \by_H-\by^H = (\by_H-\bby)-(\by^H-\bby) =\bxi_1-\bta_1 \in\bJ_H \\
\e_2 &=& \bz_h-\bz^h = (\bz_h-\bbz)-(\bz^h-\bbz) =\bxi_2-\bta_2 \in\bJ_h^H.
\end{array}\right.
\end{align}
For the sake of simplicity, we write
$$ \bxi=\bxi_1+\bxi_2,~~~\bta=\bta_1+\bta_2. $$
For the equations in $\bxi$ and $\bta$: subtract (\ref{dwfjHl}) from (\ref{dwfjH})
and subtract (\ref{dwfjHl}) from (\ref{dwfJH}) to obtain
\begin{align} \label{bxi12}
\left\{\begin{array}{rl}
(\bxi_{1,t},\bphi)+& \mu a (\bxi,\bphi)+ \int_{t_0}^t \beta(t-s) a(\bxi(s),\bphi)~ds=0 \\
& \mu a (\bxi,\bchi)+ \int_{t_0}^t \beta(t-s) a(\bxi(s), \bchi)~ds= -(\bz_{ht},\bchi),
\end{array}\right.
\end{align}
\begin{align} \label{bta12}
\left\{\begin{array}{rl}
(\bta_{1,t},\bphi) +& \mu a (\bta,\bphi)+ \int_{t_0}^t \beta(t-s) a(\bta(s), \bphi)~ds
= b(\bu_h,\bu_h,\bphi)-b(\bu^h,\bu^h,\bphi) \\
& \mu a (\bta,\bchi)+ \int_{t_0}^t \beta(t-s) a(\bta(s), \bchi)~ds= b(\bu_h,\bu_h,\bchi)
-b(\bu^h,\bu^h,\bchi),
\end{array}\right.
\end{align}
for $\bphi\in\bJ_H$ and $\bchi\in\bJ_h^H.$
\begin{lemma}\label{l2l2.bxi}
Under the assumptions of Lemma 3.2
\begin{equation}\label{l2l2.bxi1}
e^{-2\alpha t}\int_{t_0}^t e^{2\alpha \tau}\|\bxi(\tau)\|^2d\tau \le K(t)H^8.
\end{equation}
\end{lemma}

\begin{proof}
Choose $\bphi=e^{2\alpha t}\bxi_1,~\bchi=e^{2\alpha t}\bxi_2$ in (\ref{bxi12}), add the
two resulting equations and with the notation $\hxi=e^{\alpha t}\bxi$, we get
\begin{equation}\label{l2l2.bxi01}
\frac{1}{2}\frac{d}{dt}\|\hxi_1\|^2-\alpha\|\hxi_1\|^2+\mu\|\hxi\|_1^2+\int_{t_0}^t
\beta(t-s) e^{\alpha(t-s)} a(\hxi(s),\hxi)~ds \le e^{\alpha t}\|\bz_{ht}\|\|\hxi_2\|
\end{equation}
Using (\ref{jhH1}) and (\ref{bzh.est1}), we can bound the right-hand side as:
$$ \le e^{\alpha t}.K(t)H^2.cH\|\hxi_2\|_1 \le \frac{\mu\rho}{2}\|\hxi_2\|_1^2
+K(t)H^6.e^{2\alpha t}. $$
And using (\ref{jhH3}), we have
$$ -\alpha\|\hxi_1\|^2+\mu\|\hxi\|_1^2 \ge (\mu\rho-\frac{\alpha}{\lambda_1})
\|\hxi_1\|_1^2+\mu\rho\|\hxi_2\|_1^2. $$
As a result, we obtain from (\ref{l2l2.bxi01})
\begin{equation*}
\frac{d}{dt}\|\hxi_1\|^2+2(\mu\rho-\frac{\alpha}{\lambda_1})\|\hxi_1\|_1^2
+\mu\rho\|\hxi_2\|_1^2+2\int_{t_0}^t
\beta(t-s) e^{\alpha(t-s)} a(\hxi(s),\hxi)~ds \le K(t)H^6 e^{2\alpha t}.
\end{equation*}
Integrate from $t_0$ to $t$ and multiply the resulting inequality by $e^{-2\alpha t}$.
Note that the double integral turns out to be non-negative and hence
\begin{equation}\label{l2l2.bxi.2}
\|\bxi_1\|^2+e^{-2\alpha t}\int_{t_0}^t (\|\hxi_1\|_1^2+\|\hxi_2\|_1^2)~ds \le K(t)H^6.
\end{equation}
To obtain $L^2(\bL^2)$-norm estimate, we consider the following discrete backward
problem: for fixed $t_0$, let $\bw(\tau)\in \bJ_h,~\bw=\bw_1+\bw_2$ such that $\bw_1
\in \bJ_H,~\bw_2\in \bJ_h^H$ be the unique solution of ($t_0\le \tau <t$)
\begin{equation}\label{bp1} 
\left\{\begin{array}{rcl}
(\bphi,\bw_{1,\tau}) &-&\mu a(\bphi,\bw)-\int_{\tau}^t \beta(s-\tau) a(\bphi,\bw(s))
~ds = e^{2\alpha t}(\bphi,\bxi_1) \\
&-& a(\bchi,\bw)-\int_{\tau}^t \beta(s-\tau) a(\bchi,\bw(s))~ds= e^{2\alpha t}
(\bchi,\bxi_2) \\
&& \bw_1(t)=0.
\end{array}\right.
\end{equation}
With change of variable, we can make it a forward problem and it turns out to be
a linearized version of (\ref{dwfJH}) and hence is well-posed. As in \cite{GP11},
we can easily obtain the following regularity result.
\begin{equation}\label{bp1.est}
\int_{t_0}^t e^{-2\alpha \tau}\|\bw\|_2^2d\tau \le C\int_{t_0}^t \|\hxi\|^2d\tau.
\end{equation}
Now, choose $\bphi=\bxi_1,~\bchi=\bxi_2$ and use (\ref{bxi12}) with $\bphi=\bw_1, ~\bchi=\bw_2$ to find that
\begin{align*}
\|\hxi(\tau)\|^2 &= (\bxi_1,\bw_{1,\tau})-\mu a (\bxi,\bw)-\int_{\tau}^t \beta(s-\tau)
a(\bxi,\bw(s))~ds \\
\le \frac{d}{dt} (\bxi_1,\bw_1) &+\int_{t_0}^\tau \beta(\tau-s) a(\bxi(s),\bw)~ds
-\int_{\tau}^t \beta(s-\tau) a(\bxi,\bw(s))~ds+(\bz_{ht},\bw_2).
\end{align*}
Integrate from $t_0$ to $t$ and note that the double integral terms cancel each other.
\begin{equation}\label{l2l2.bxi03}
\int_{t_0}^t \|\hxi(\tau)\|^2d\tau =((\bxi_1(t),\bw_1(t))-(\bxi_1(t_0),\bw_1(t_0))
+\int_{t_0}^t (\bz_{ht},\bw_2) d\tau.
\end{equation}
But $\bw_1(t)=0$ and $\bxi_1(t_0)=\by_H(t_0)-\bby(t_0)=0$. Next, we observe that
$$ \bw_1\in\bJ_H \implies P_H\bw_1=\bw_1,~~\bw_2\in\bJ_h^H \implies P_H\bw_2=0. $$
As a result
$$ \bw_1-P_H\bw_1=0,~~\bw_2-P_H\bw_2=\bw_2,~~\mbox{or}~~\bw-P_H\bw=\bw_2. $$
Therefore,
\begin{align}\label{new.est}
(\bz_{ht},\bw_2) &=(\bz_{ht},\bw-P_H\bw) \nonumber \\
& \le \|\bz_{ht}\|\|\bw-P_H\bw\| \le K(t)H^2.cH^2\|\bw\|_2.
\end{align}
From (\ref{l2l2.bxi03}), we get
$$ \int_{t_0}^t \|\hxi(\tau)\|^2d\tau \le K(t) e^{2\alpha t}H^4\Big(\int_{t_0}^t
e^{-2\alpha \tau}\|\bw\|_2^2d\tau\Big)^{1/2}. $$
Use (\ref{bp1.est}) to conclude.
\end{proof}
In order to obtain optimal $L^{\infty}(\bL^2)$ estimate, we would like to introduce
Stokes-Volterra type projections $(S_H,S_h^H)$ for $t>t_0$ defined as below:
$$ S_H:\bJ_h\to \bJ_H,~~S_h^H:\bJ_h\to \bJ_h^H, $$
and with the notations
$$\bzt_1 :=\by_H-S_H\bu_h \in \bJ_H,~~\bzt_2 :=\bz_h-S_h^H\bu_h\in\bJ_h^H $$
the following system is satisfied.
\begin{align} \label{bzt12}
\left\{\begin{array}{rl}
\mu a (\bzt,\bphi)+ \int_{t_0}^t \beta(t-s) a(\bzt(s),\bphi)~ds=0,~\bphi\in\bJ_H, \\
\mu a (\bzt,\bchi)+ \int_{t_0}^t \beta(t-s) a(\bzt(s), \bchi)~ds= -(\bz_{ht},\bchi),
~\bchi\in\bJ_h^H.
\end{array}\right.
\end{align}
For the sake of convenience, we have written $\bzt=\bzt_1+\bzt_2.$
\begin{lemma}\label{l2.bzt}
Under the assumptions of Lemma 3.2
\begin{equation}\label{l2.bzt1}
\|\bzt\|+\|\bzt_t\| \le K(t)H^4.
\end{equation}
\end{lemma}

\begin{proof}
Choose $\bphi= e^{2\alpha t}\bzt_1,~\bchi=e^{2\alpha t}\bzt_2$ to obtain
\begin{equation}\label{l2.bzt01}
 \mu\|\hzt\|_1^2+\int_{t_0}^t \beta(t-s) e^{\alpha(t-s)} a(\hzt(s),\hzt)
\le e^{\alpha t}\|\bz_{ht}\|\|\hzt_2\|.
\end{equation}
As in (\ref{l2l2.bxi01})-(\ref{l2l2.bxi.2}), we establish
\begin{equation}\label{l2.bzt2}
e^{-2\alpha t}\int_{t_0}^t \|\hzt\|_1^2ds \le K(t)H^6.
\end{equation}
Now from (\ref{l2.bzt01}), we have
$$ \mu\|\hzt\|_1^2 \le \|\hzt\|_1\int_{t_0}^t \beta(t-s) e^{\alpha(t-s)}
\|\hzt(s)\|_1~ds+\|\bz_{ht}\|\|\hzt_2\|. $$
Use (\ref{l2.bzt2}) to conclude that
\begin{equation}\label{l2.bzt3}
\|\bzt\|_1 \le K(t)H^3.
\end{equation}
In order to obtain optimal $l^{\infty}(\bL^2)$-norm estimate, we would use
Aubin-Nitsche duality argument. For that purpose, we consider the following
Galerkin approximation of steady Stoke problem: let $\bw_h\in$ be the solution of
$$ \mu a(\bv,\bw_h)=(\bv,\hzt_1+\hzt_2),~\bv\in\bJ_h. $$
Writing $\bw_1=P_H\bw_h,~\bw_2=(I-P_H)\bw_h$, we split the above equation as
\begin{equation}\label{dual.prob1}
\left\{\begin{array}{rcl}
\mu a(\bphi,\bw_h) &=& (\bphi,\hzt_1),~\bphi\in\bJ_H, \\
\mu a(\bchi,\bw_h) &=& (\bchi,\hzt_2),~\bchi\in \bJ_h^H.
\end{array}\right.
\end{equation}
It is easy to establish the following regularity result.
\begin{equation}\label{dual.reg1}
\|\bw_h=\bw_1+\bw_2\|_2 \le c\|\hzt_1+\hzt_2\|. 
\end{equation}
Now, put $\bphi=\hzt_1,~\bchi=\hzt_2$ in (\ref{dual.prob1}) and use (\ref{bzt12})
with $\bphi=\bw_1,~\bchi=\bw_2$ to find that
\begin{align*}
\|\hzt\|^2 &=\mu a(\hzt,\bw_h) \\
&= -e^{\alpha t}(\bz_{ht},\bw_2)-\int_{t_0}^t \beta(t-s) e^{\alpha(t-s)} a(\hzt(s),\bw_h).
\end{align*}
Use the fact that $\td_h\bw_h =\hzt_1+\hzt_2$ to obtain
\begin{equation}\label{l2.bzt02}
\|\hzt\|^2+\int_{t_0}^t \beta(t-s)e^{\alpha(t-s)}(\hzt(s),\hzt)~ds =
 -e^{\alpha t}(\bz_{ht},\bw_2).
\end{equation}
We can estimate the right-hand side as in (\ref{new.est}). Integrate and observe that the double integral is non-negative to conclude that	
\begin{equation}\label{l2.bzt4}
e^{-2\alpha t}\int_{t_0}^t \|\hzt\|^2 ds \le K(t)H^8.
\end{equation}
Now from (\ref{l2.bzt02}), we find that
$$ \|\hzt\|^2 \le e^{\alpha t}|(\bz_{ht},\bw_2)|+\|\hzt\|\int_{t_0}^t \beta(t-s) e^{\alpha(t-s)}\|\hzt\|~ds. $$
Use (\ref{new.est}) and (\ref{l2.bzt4}) to establish
$$ \|\bzt\| \le K(t)H^4. $$
For the remaining part, we differentiate (\ref{bzt12}).
\begin{align} \label{bztt12}
\left\{\begin{array}{rl}
\mu a (\bzt_t,\bphi)+\beta(0)a(\bzt,\bphi)+\int_{t_0}^t \beta_t(t-s) a(\bzt(s),\bphi)~ds=0,~\bphi\in\bJ_H, \\
\mu a (\bzt_t,\bchi)+\beta(0)a(\bzt,\bchi)+\int_{t_0}^t \beta_t(t-s)
a(\bzt(s), \bchi)~ds= -(\bz_{htt},\bchi),~\bchi\in\bJ_h^H.
\end{array}\right.
\end{align}
Choose $\bphi= e^{2\alpha t}\bzt_{1,t},~\bchi=e^{2\alpha t}\bzt_{2,t}$ to obtain
\begin{equation*}
\mu e^{2\alpha t}\|\bzt_t\|_1^2 \le \beta(0) e^{\alpha t}\|\hzt\|_1\|\bzt_t\|_1
+\delta e^{\alpha t}\|\bzt_t\|_1\int_{t_0}^t \beta(t-s) e^{\alpha(t-s)} \|\hzt(s)\|_1ds
+e^{\alpha t}\|\bz_{htt}\|\|\hzt_{2,t}\|.
\end{equation*}
Use kickback argument to find that
$$ \|\bzt_t\|_1^2 \le c\|\hzt\|_1^2+ce^{-2\alpha t}\int_{t_0}^t \|\hzt(s)\|_1^2ds
+K(t)H^6. $$
From (\ref{l2.bzt2})-(\ref{l2.bzt3}), we conclude
\begin{equation}\label{l2.bzt5}
\|\bzt_t\|_1 \le K(t)H^3.
\end{equation}
For $\bL^2$-norm estimate, we again make use of Aubin-Nitsche duality argument.
Similar to the problem (\ref{dual.prob1}), we consider
\begin{equation}\label{dual.prob2}
\left\{\begin{array}{rcl}
\mu a(\bphi,\bw_h) &=& (\bphi,e^{\alpha t}\bzt_{1,t}),~\bphi\in\bJ_H, \\
\mu a(\bchi,\bw_h) &=& (\bchi,e^{\alpha t}\bzt_{2,t}),~\bchi\in\bJ_h^H.
\end{array}\right.
\end{equation}
It is easy to establish the following regularity result.
\begin{equation}\label{dual.reg2}
\|\bw_h=\bw_1+\bw_2\|_2 \le ce^{\alpha t}\|\bzt_{1,t}+\bzt_{2,t}\|.
\end{equation}
Now, put $\bphi=e^{\alpha t}\bzt_{1,t},~\bchi=e^{\alpha t}\bzt_{2,t}$ in (\ref{dual.prob2}) and use (\ref{bztt12}) with $\bphi=e^{\alpha t}\bw_1,~\bchi=e^{\alpha t}\bw_2$ to find that
\begin{align*}
e^{2\alpha t}\|\bzt_t\|^2 &=\mu e^{\alpha t} a(\bzt_t,\bw_h) \\
&= -e^{\alpha t}(\bz_{htt},\bw_2)-\beta(0) a(\hzt,\bw_h)+\delta\int_{t_0}^t \beta(t-s) e^{\alpha(t-s)} a(\hzt(s),\bw_h) \\
& \le \Big\{e^{\alpha t}\|\bz_{htt}\|+c\|\hzt\|+c\|\hzt\|\int_{t_0}^t e^{-(\delta-\alpha)(t-s)} ds\Big\}\|\bw_h\|_2.
\end{align*}
Use (\ref{dual.reg2}) and the estimate for $\|\bzt\|$ to establish
$$ \|\bzt_t\| \le K(t)H^4. $$
\end{proof}
Now we are in a position to estimate $L^{\infty}(\bL^2)$-norm of $\bxi$, that is,
of $\bxi_1$ and $\bxi_2$. Using the definitions of $\bxi_i,\bzt_i,~i=1,2$, we write
\begin{equation*}
\left\{\begin{array}{rcl}
\bxi_1 &=& \by_H-\bby= (\by_H-S_H\bu_h)-(\bby-S_H\bu_h) =:\bzt_1-\bth_1, \\
\bxi_2 &=& \bz_h-\bbz= (\bz_h-S_h^H\bu_h)-(\bbz-S_h^H\bu_h) =:\bzt_2-\bth_2.
\end{array}\right.
\end{equation*}
From (\ref{bxi12}) and (\ref{bzt12}), we have
\begin{equation}\label{bth12}
\left\{\begin{array}{rcl}
(\bth_{1,t},\bphi) &+& \mu a(\bth,\bphi)+\int_{t_0}^t \beta(t-s) a(\bth(s),\bphi)~ds
=(\bzt_{1,t},\bphi),~\bphi\in\bJ_H, \\
&&\mu a(\bth,\bchi)+\int_{t_0}^t \beta(t-s) a(\bth(s),\bchi)~ds=0,~\bchi\in\bJ_h^H.
\end{array}\right.
\end{equation}
\begin{lemma}\label{l2.bxi}
Under the assumptions of Lemma 3.2
$$ \|\bxi\| \le K(t)H^4. $$
\end{lemma}

\begin{proof}
Put $\bphi=e^{2\alpha t}\bth_1,~\bchi=e^{2\alpha t}\bth_2$ in (\ref{bth12}) to find
\begin{equation}\label{l2.bxi01}
\frac{1}{2}\frac{d}{dt}\|\hth_1\|^2-\alpha\|\hth_1\|^2+\mu\|\hth\|_1^2+\int_{t_0}^t
\beta(t-s) e^{\alpha(t-s)} a(\hth(s),\hth)~ds \le e^{\alpha t}\|\bzt_{1,t}\|\|\hth_1\|
\end{equation}
We recall that the spaces $J_H$ and $\bJ_h^H$ are orthogonal in $\bL^2$-inner product.
That is,
$$ \mbox{for}~\bphi\in\bJ_H,~\bchi\in\bJ_h^H,~~(\bphi,\bchi)=0. $$
Hence
$$ \|\hth_1\|^2 \le \|\hth_1\|^2+\|\hth_2\|^2= \|\hth\|^2 \le \|\hxi\|^2+\|\hzt\|^2,
~~\|\bzt_{1,t}\|^2 \le \|\bzt_t\|^2. $$
And
$$ -\alpha\|\hth_1\|^2+\mu\|\hth\|_1^2 =(\mu-\alpha\lambda_1)\|\hth_1\|_1^2
+\mu\|\hth_2\|_1^2. $$
As a result, after integrating (\ref{l2.bxi01}) with respect to time from $t_0$ to
$t$, we obtain
\begin{equation}\label{l2.bxi02}
\|\hth_1\|^2+\int_{t_0}^t \big(\|\hth_1\|_1^2+\|\hth_2\|_1^2\big)~ds \le \Big(\int_{t_0}^t
e^{2\alpha t}\|\bzt_t\|^2ds\Big)^{1/2}\Big(\int_{t_0}^t \big(\|\hxi\|^2
+\|\hzt\|^2\big)~ds\Big)^{1/2}.
\end{equation}
As usual we have dropped the resulting double integral as it is non-negative.
We now use Lemmas \ref{l2l2.bxi} and \ref{l2.bzt} to conclude from (\ref{l2.bxi02}) that
\begin{equation}\label{l2.bxi2}
\|\hth_1\|^2+e^{-2\alpha t}\int_{t_0}^t e^{2\alpha s}\big(\|\bth_1\|_1^2 +\|\bth_2\|_1^2\big)~ds \le K(t)H^8.
\end{equation}
We again choose $\bchi=e^{2\alpha t}\bth_2$ in (\ref{bth12}) to find
\begin{align*}
\mu \|\hth_2\|_1^2= -\mu a(\hth_1,\hth_2)-\int_{t_0}^t \beta(t-s) e^{\alpha(t-s)}
a(\hth(s),\hth_2)~ds
\end{align*}
Using kickback argument, we obtain
\begin{align*}
\|\hth_2\|_1 \le \|\hth_1\|_1+c\Big(\int_{t_0}^t \big(\|\hth_1\|^2
+\|\hth_2\|^2\big)~ds\Big)^{1/2}.
\end{align*}
Since $\bth_1\in\bJ_H$, we use inverse inequality and (\ref{l2.bxi2}) to note that
$$ \|\bth_1\|_1 \le cH^{-1}\|\bth\| \le K(t)H^3. $$
Hence, we conclude that
$$ \|\bth_2\|_1 \le K(t)H^3. $$
Now use (\ref{jhH1}) to see that
\begin{equation}\label{l2.bxi3}
\|\bth_2\| \le K(t)H^4.
\end{equation}
Combining (\ref{l2.bxi2})-(\ref{l2.bxi3}), we establish
$$ \|\bth\| \le K(t)H^4. $$
Use triangle inequality and the estimates of $\bzt$ and $\bth$ to complete the proof.
\end{proof}
\noindent
We are now left with the estimate of $\bta$, the error due to the nonlinearity.
\begin{lemma}\label{l2.err}
Under the assumptions of Lemma 3.2 and that $H$ is small enough to satisfy (\ref{H.restrict}) and
$$ \mu\rho-2H\|\bbu\|_2 \ge 0,~~\mu-H(\|\bbu\|_2+\|\by^H\|_2) \ge 0, $$
we have
$$ \|(\bu_h-\bu^h)(t)\| \le K(t)H^4. $$
\end{lemma}

\begin{proof}
We choose $\bphi=e^{2\alpha t}\bta_1,~\bchi=e^{2\alpha t}\bta_2$ in (\ref{bta12}).
\begin{align}\label{l2.err01}
\frac{1}{2}\frac{d}{dt}\|\hta_1\|^2+\mu \|\hta\|_1^2+\int_{t_0}^t \beta(t-s) e^{\alpha
(t-s)} a(\hta(s),\hta)~ds = e^{2\alpha t}\Lambda_h(\bta_1,\bta_2),
\end{align}
where
$$ \Lambda_h(\bta_1,\bta_2)=\Lambda_{h,1}(\bta_1)+\Lambda_{h,2}(\bta_2), $$
and
\begin{align*}
\Lambda_{h,1}(\bta_1) &=b(\bu_h,\bu_h,\bta_1)-b(\bu^h,\bu^h,\bta_1) \\
&= b(\bxi-\bta,\bu_h,\bta_1)+b(\bu_h,\bxi-\bta,\bta_1)-b(\bxi-\bta,\bxi-\bta,\bta_1) \\
\Lambda_{h,2}(\bta_2) &=b(\bu_h,\bu_h,\bta_2)-b(\bu^h,\bu^h,\bta_2) \\
&= b(\bxi-\bta,\bu_h,\bta_2)+b(\bu_h,\bxi-\bta,\bta_2)-b(\bxi-\bta,\bxi-\bta,\bta_2).
\end{align*}
Therefore
\begin{align}\label{l2.err02}
\Lambda_h(\bta_1,\bta_2)= b(\bxi-\bta,\bbu,\bta)+b(\bu_h,\bxi,\bta).
\end{align}
\noindent We estimate the nonlinear terms as follows:
\begin{align*}
b(\bu_h,\bxi,\bta)&+b(\bxi-\bta_1,\bbu,\bta) \le \big\{\|\bxi\|\|\bu_h\|_2+
(\|\bxi\|+\|\bta_1\|)\|\bbu\|_2\big\}\|\bta\|_1 , \\
b(\bta_2,\bbu,\bta) &\le \|\bta_2\|_1\|\bbu\|_2(\|\bta_1\|+\|\bta_2\|) \le
\|\bta_2\|_1\|\bbu\|_2\|\bta_1\|+H\|\bbu\|_2\|\bta_2\|_1^2.
\end{align*}
Therefore, for $\epsilon,\epsilon_1>0$,
\begin{align*}
\Lambda_h(\bta_1,\bta_2) &\le \epsilon\|\bta\|_1^2+\epsilon_1\|\bta_2\|_1^2+c(\epsilon) (\|\bu_h\|_2^2+\|\bbu\|_2^2)\|\bxi\|^2 \\
& ~~+c(\epsilon,\epsilon_1)\|\bbu\|_2^2\|\bta_1\|^2+H\|\bbu\|_2\|\bta_2\|_1^2.
\end{align*}
Now, from (\ref{l2.err01}), we find that
\begin{align}\label{l2.err03}
\frac{d}{dt}\|\hta_1\|^2+2\mu\rho &(\|\hta_1\|_1^2+\|\hta_2\|_1^2)+2\int_{t_0}^t \beta(t-s) e^{\alpha(t-s)} a(\hta(s),\hta)~ds \le 2\epsilon\|\hta\|_1^2+2\epsilon_1\|\hta_2\|_1^2
\nonumber \\
&+c(\epsilon) (\|\bu_h\|_2^2+\|\bbu\|_2^2)\|\hxi\|^2+c(\epsilon,\epsilon_1) \|\bbu\|_2^2\|\hta_1\|^2+2H\|\bbu\|_2\|\hta_2\|_1^2.
\end{align}
We choose $\epsilon=\epsilon_1=\mu\rho$ and assume that $H$ small enough such that
$$ \mu\rho-2H\|\bbu\|_2 \ge 0 $$
to obtain after integration
$$ \|\bta_1\|^2+e^{-2\alpha t}\int_{t_0}^t ((\|\hta_1\|_1^2+\|\hta_2\|_1^2))~ds \le
K(t)H^8+K\int_{t_0}^t \|\bta_1(s)\|^2ds. $$
Apply Gronwall's lemma to establish $L^{\infty}(\bL^2)$-norm estimate of $\bta_1$.
We note that
$$ \|\bta_1\|_1 \le cH^{-1}\|\bta_1\| \le K(t)H^3. $$
For $\bta_2$, we again put $\bchi=e^{2\alpha t}\bta_2$ in (\ref{bta12}).
\begin{align}\label{l2.err04}
\mu \|\hta_2\|_1^2= +e^{2\alpha t}\Lambda_{h,2}(\bta_2)-\mu a(\hta_1,\hta_2)
-\int_{t_0}^t \beta(t-s) e^{\alpha (t-s)} a(\hta(s),\hta_2)~ds.
\end{align}
Recall that
$$ \Lambda_{h,2}(\bta_2)= b(\bxi-\bta,\bbu,\bta_2)+b(\bu_h,\bxi-\bta_1,\bta_2) +b(\bxi-\bta,\bta_1,\bta_2). $$
And
\begin{align*}
b(\bxi-\bta_1,\bbu,\bta_2)+b(\bu_h,\bxi-\bta_1,\bta_2) \le (\|\bxi\|+\|\bta_1\|)
(\|\bbu\|_2+\|\bu_h\|_2)\|\bta_2\|_1 \\
b(\bxi-\bta_1,\bta_1,\bta_2) \le (\|\bxi\|_1+\|\bta_1\|_1)\|\bta_1\|_1\|\bta_2\|_1 \\
b(-\bta_2,\bbu,\bta_2)+b(-\bta_2,\bta_1,\bta_2) \le H(\|\bbu\|_2+\|\bta_1\|_2)\|\bta_2\|_1^2.
\end{align*}
Note that $\|\bta_1\| \le \|\by^H\|+\|\bby\|.$ And under the assumption
$$ \mu-H(\|\bbu\|_2+\|\by^H\|_2) \ge 0 $$
we easily obtain that
$$ \|\bta_2\|_1 \le K(t)H^3 $$
and hence
$$ \|\bta_2\| \le cH\|\bta_2\|_1 \le K(t)H^4. $$
Now, triangle inequality completes the proof.
\end{proof}

\subsection{NLG II}

In this subsection, we deal with the error estimate for NLG II.
As earlier, we split the error in two, that is, $\e=\bu_h-\bu^h=\bxi-\bta$. The equations and hence the estimates regarding $\bxi$ remain same and are optimal in nature. The equation in $\bta$ reads as follows:
\begin{align} \label{bta12a}
\left\{\begin{array}{rl}
(\bta_{1,t},\bphi)+\mu a (\bta,\bphi)+ \int_{t_0}^t \beta(t-s) a(\bta(s), \bphi)~ds
=& b(\bu_h,\bu_h,\bphi)-b(\bu^h,\bu^h,\bphi) \\
\mu a (\bta,\bchi)+ \int_{t_0}^t \beta(t-s) a(\bta(s), \bchi)~ds=& b(\bu_h,\bu_h,\bchi)
-b(\bu^h,\bu^h,\bchi) \\
&+b(\bz^h,\bz^h,\bchi),
\end{array}\right.
\end{align}
for $\bphi\in\bJ_H$ and $\bchi\in\bJ_h^H.$
\begin{lemma}\label{l2.err1}
Under the assumptions of Lemma \ref{l2.err}, we have
$$ \|(\bu_h-\bu^h)(t)\| \le K(t)H^3. $$
\end{lemma}

\begin{proof}
We choose $\bphi=e^{2\alpha t}\bta_1,~\bchi=e^{2\alpha t}\bta_2$ in (\ref{bta12a}).
\begin{align}\label{l2.err101}
\frac{1}{2}\frac{d}{dt}\|\hta_1\|^2+\mu \|\hta\|_1^2+\int_{t_0}^t \beta(t-s) e^{\alpha
(t-s)} a(\hta(s),\hta)~ds = e^{2\alpha t}\big\{\Lambda_h(\bta_1,\bta_2)
+b(\bz^h,\bz^h,\bta_2)\big\},
\end{align}
Since $\bz^h=\bz_h+\bta_2-\bxi_2,$ we have
$$ b(\bz^h,\bz^h,\bta_2)=b(\bz_h+\bta_2-\bxi_2,\bz_h-\bxi_2,\bta_2). $$
Now
\begin{align*}
b(\bxi_2,\bz_h-\bxi_2,\bta_2) &\le \|\bxi_2\|_1(\|\bz_h\|_1+\|\bxi_2\|_1)\|\bta_2\|_1 \\
b(\bta_2,\bz_h-\bxi_2,\bta_2) &=b(\bta_2,\bbz,\bta_2) \le cH\|\bbz\|_2\|\bta_2\|_1^2 \\
b(\bz_h,\bz_h-\bxi_2,\bta_2) & \le  \|\bz_h\|^{1/2}\|\bz_h\|_1^{1/2}(\|\bz_h\|_1
+\|\bxi_2\|_1)\|\bta_2\|^{1/2}\|\bta_2\|_1^{1/2} \\
&+\|\bz_h\|^{1/2}\|\bz_h\|_1^{1/2}\|\bta_2\|_1(\|\bz_h\|^{1/2}\|\bz_h\|_1^{1/2}
+\|\bxi_2\|_1) \\
\le cH\|\bz_h\|_1 & (\|\bz_h\|_1+\|\bxi_2\|_1)\|\bta_2\|_1+cH^{1/2}\|\bz_h\|_1\|\bta_2\|_1.
cH^{1/2}(\|\bz_h\|_1+\|\bxi_2\|_1) \\
\le cH\|\bz_h\|_1^2 & \|\bta_2\|_1+cH\|\bz_h\|_1\|\bxi_2\|_1\|\bta_2\|_1.
\end{align*}
As earlier, for small $H$, we obtain
$$ \|\bta_1\|^2+e^{-2\alpha t}\int_{t_0}^t ((\|\hta_1\|_1^2+\|\hta_2\|_1^2))~ds \le
K(t)H^8+K(t)H^2\|\bz_h\|_1^4+K\int_{t_0}^t \|\bta_1(s)\|^2ds, $$
which results in
$$ \|\bta_1\|^2+e^{-2\alpha t}\int_{t_0}^t ((\|\hta_1\|_1^2+\|\hta_2\|_1^2))~ds \le
K(t)H^6. $$
That is
$$ \|\bta_1\| \le K(t)H^3,~~\|\bta_1\|_1 \le K(t)H^2. $$
As in the previous section, using only the second equation of (\ref{bta12a}) we can easily conclude that
$$ \|\bta_2\| \le K(t)H^3,~~\|\bta_2\|_1 \le K(t)H^2. $$
\end{proof}

\noindent
\begin{remark}\label{rmk}
The analysis reveals that the decrease in the order of convergence is due to the presence of $b(\bz_h,\bz_h,\bchi)$ in the error equation. So, whereas in NLG I, we keep the nonlinearity
in both the equations, in NLG II, the second equation is made linear in $\bz^h$ by dropping
the term $b(\bz_h,\bz_h,\bchi)$ and which in turn appears in the error equation and is responsible for bringing down the rate of convergence in the above analysis.
\end{remark}

\subsection{Improved Error Estimate}

In this section, we try to improve the rate of convergence, using the technique of
\cite{MX95}. The same technique is applicable for NSE also, but it is not
straightforward, as the estimate of the function $f(\bu)$ in their semi-linear problem
does not always hold for our $f(\bu)$ (which involves the non-linear term) and we have to be careful in order to obtain similar results.

\noindent
First, we note that the second equation of (\ref{dwfJH1}) can be written as
\begin{equation}\label{relbyz} 
\bz^h=\Phi(\by^H),
\end{equation}
where $\Phi:\bJ_H\to \bJ_h^H$. Using this, we can write the equation in $\Phi(\by_H)$,
for $\bchi\in\bJ_h^H$.
\begin{align}\label{eqn.phi1}
\mu a (\by_H+\Phi(\by_H),\bchi) &+ b(\by_H+\Phi(\by_H),\by_H,\bchi)+b(\by_H,\Phi(\by_H), \bchi) \nonumber \\
&+ \int_{t_0}^t \beta(t-s) a((\by_H+\Phi(\by_H))(s),\bchi)~ds =(\f, \bchi).
\end{align}
\begin{lemma}\label{err.phi} 
Under the assumptions of Lemma 3.2 and that $H$ is small enough to satisfy
$$ \frac{\mu}{2}-cH\|\bu_h\|_1 \ge 0, $$
we have
\begin{equation}\label{ep1}
\|\bz_h-\Phi(\by_H)\|+H\|\bz_h-\Phi(\by_H)\|_1 \le K(t)H^4.
\end{equation}
\end{lemma}

\begin{proof}
With the notation $\Phi_{\e}:= \bz_h-\Phi(\by_H)\in \bJ_h^H$, we have, by deducting
(\ref{eqn.phi1}) from the second equation of (\ref{dwfjH})
\begin{align}\label{eqn.pe}
\mu a(\Phi_{\e},\bchi)+\int_{t_0}^t \beta(t-s) a(\Phi_{\e}(s),\bchi)~ds=& -(\bz_{h,t},
\bchi)-b(\bu_h,\bu_h,\bchi)+b(\by_H,\Phi(\by_H),\bchi) \nonumber \\
&+b(\by_H+\Phi(\by_H),\by_H,\bchi).
\end{align}
Put $\bchi=\Phi_{\e}$ to obtain
\begin{align}\label{ep01}
\mu \|\Phi_{\e}\|_1^2+\int_{t_0}^t \beta(t-s) a(\Phi_{\e}(s),\Phi_{\e})~ds = &
-(\bz_{h,t},\Phi_{\e})-b(\Phi_{\e},\bu_h,\Phi_{\e}) \nonumber \\
&-b(\Phi(\by_H),\Phi(\by_H),\Phi_{\e}).
\end{align}
Note that
\begin{align*}
-(\bz_{h,t},\Phi_{\e}) &\le \|\bz_{h,t}\|\|\Phi_{\e}\| \le K(t)H^3\|\Phi_{\e}\|_1 \\
-b(\Phi_{\e},\bu_h,\Phi_{\e}) &\le c\|\bu_h\|_1\|\Phi_{\e}\|\|\Phi_{\e}\|_1
\le cH\|\bu_h\|_1\|\Phi_{\e}\|_1^2 \\
-b(\Phi(\by_H),\Phi(\by_H),\Phi_{\e}) &= -b(\bz_h-\Phi_{\e},\bz_h,\Phi_{\e}) \\
&\le (\|\bz_h\|^{1/2}\|\bz_h\|_1^{1/2}+\|\Phi_{\e}\|^{1/2}\|\Phi_{\e}\|_1^{1/2})
\|\bz_h\|_1\|\Phi_{\e}\|^{1/2}\|\Phi_{\e}\|_1^{1/2} \\
&\le K(t)H^3\|\Phi_{\e}\|_1+cH\|\bz_h\|_1\|\Phi_{\e}\|_1^2.
\end{align*}
Therefore, from (\ref{ep01}), we find
\begin{align}\label{ep02}
\frac{\mu}{2} \|\Phi_{\e}\|_1^2+\int_{t_0}^t \beta(t-s) a(\Phi_{\e}(s),\Phi_{\e})~ds \le
K(t)H^6+cH\|\bu_h\|_1\|\Phi_{\e}\|_1^2
\end{align}
We have used the fact that $\|\bz_h\|_1 \le \|\bu_h\|_1+\|\by_H\|_1 \le c\|\bu_h\|_1.$ And assuming
$H$ to be small enough to satisfy
$$ \frac{\mu}{2}-cH\|\bu_h\|_1 \ge 0 $$
we establish after integrating (\ref{ep02})
$$ \int_{t_0}^t \|\Phi_{\e}\|_1^2ds \le K(t)H^6. $$
Use this result to estimate the integral term on (\ref{ep02}) to conclude
$$ \|\Phi_{\e}\|_1 \le K(t)H^3. $$
And hence
$$ \|\Phi_{\e}\| \le cH\|\Phi_{\e}\|_1 \le K(t)H^4. $$
This completes the proof.
\end{proof}

\begin{lemma}\label{err.e2}
Under the assumptions Lemma \ref{err.phi}, we have
\begin{eqnarray}
\|\e_2\|_1^2 \le K(t)H^6+c\|\e_1\|_1^2+cte^{-2\alpha t}\int_{t_0}^t e^{2\alpha s}\|\e_1(s)\|_1^2 ds \label{ee2.1} \\
\|\e_2\|^2 \le K(t)H^8+cH^2\|\e_1\|_1^2+cH^2te^{-2\alpha t}\int_{t_0}^t
e^{2\alpha s}\|\e_1(s)\|_1^2 ds. \label{ee2.2}
\end{eqnarray}
\end{lemma}

\begin{proof}
Recall that $\e_2=\bz_h-\bz^h= (\bz_h-\Phi(\by_H))-(\bz^h-\Phi(\by_H))$. With the
notation $\Phi^{\e}=\bz^h-\Phi(\by_H)$, we have $\e_2=\Phi_{\e}-\Phi^{\e}.$ The
equation in $\Phi^{\e}$ can be obtained by deduction (\ref{eqn.phi1}) from the second
equation of (\ref{dwfJH1}).
\begin{align}\label{eqn.pe1}
\mu a(\Phi^{\e}-\e_1,\bchi)+\int_{t_0}^t \beta(t-s) a(\Phi^{\e}(s)-\e_1,\bchi)~ds=& 
-b(\bu^h,\by^H,\bchi)-b(\by^H,\bz^h,\bchi) \nonumber \\
+b(\by_H,\Phi(\by_H),\bchi)&+b(\by_H+\Phi(\by_H),\by_H,\bchi).
\end{align}
Put $\bchi=\Phi^{\e}$ to obtain
\begin{align}\label{ee201}
\mu \|\Phi^{\e}\|_1^2+\int_{t_0}^t \beta(t-s) a(\Phi^{\e}(s),\Phi^{\e})~ds= \mu a(\e_1,
\Phi^{\e})+\int_{t_0}^t \beta(t-s) a(\e_1(s),\Phi^{\e})~ds \nonumber \\
+b(\bu^h,\e_1,\Phi^{\e})+b(\e_1-\Phi^{\e},\by_H,\Phi^{\e})+b(\e_1,\Phi(\by_H),\Phi^{\e})
\end{align}
Note that
\begin{align*}
b(\bu^h,\e_1,\Phi^{\e})&=b(\bu_h-\e_1-\Phi_{\e}+\Phi^{\e},\e_1,\Phi^{\e}) \\
&\le \|\bu_h\|_2\|\e_1\|\|\Phi^{\e}\|_1+\|\e_1\|_1^2\|\Phi^{\e}\|_1+\|\Phi_{\e}\|_1
\|\e_1\|_1\|\Phi^{\e}\|_1+cH\|\e_1\|_1\|\Phi^{\e}\|_1^2 \\
& (\mbox{for the last estimate, we have used}~\|\Phi^{\e}\|_{\bL^4} \le
cH^{1/2}\|\Phi_{\e}\|_1) \\
b(\e_1-\Phi^{\e},\by_H,&\Phi^{\e})=b(\e_1-\Phi^{\e},\bu_h-\bz_h,\Phi^{\e}) \\
&\le \|\e_1\|\|\bu_h\|_2\|\Phi^{\e}\|_1+\|\e_1\|_1\|\bz_h\|_1\|\Phi^{\e}\|_1
+cH\|\by_H\|_1\|\Phi_{\e}\|_1^2 \\
b(\e_1,\Phi(\by_H),\Phi^{\e}) &\le \|\e_1\|_1(\|\Phi_{\e}\|_1+\|\bz_h\|_1)\|\Phi^{\e}\|_1.
\end{align*}
Therefore, we find from (\ref{ee201})
\begin{align}\label{ee202}
(\mu-cH\|\by_H\|_1) \|\Phi^{\e}\|_1^2+2\int_{t_0}^t \beta(t-s) a(\Phi^{\e}(s), &
\Phi^{\e})~ds \le c\|\e_1\|^2+c(1+H^2)\|\e_1\|_1^2 \nonumber \\
& +ce^{-2\alpha t}\int_{t_0}^t e^{2\alpha s}\|\e_1(s)\|_1^2 ds.
\end{align}
Assuming $H$ small enough to satisfy
$$ \mu-cH\|\by_H\|_1 \ge 0 $$
we have, after integration
\begin{equation}\label{ee203}
e^{-2\alpha t}\int_{t_0}^t e^{2\alpha s} \|\Phi^{\e}\|_1^2 \le K(t)H^6+ce^{-2\alpha t}\int_{t_0}^t e^{2\alpha s} \|\e_1(s)\|_1^2 ds.
\end{equation}
Use (\ref{ee203}) to estimate the integral term in (\ref{ee202}) to obtain
$$ \|\Phi^{\e}\|_1^2 \le K(t)H^6+c\|\e_1\|_1^2+cte^{-2\alpha t}\int_{t_0}^t e^{2\alpha s}\|\e_1(s)\|_1^2 ds. $$
And so
$$ \|\Phi^{\e}\|^2 \le K(t)H^8+cH^2\|\e_1\|_1^2+cH^2te^{-2\alpha t}\int_{t_0}^t
e^{2\alpha s}\|\e_1(s)\|_1^2 ds. $$
Using triangle inequality, we complete the proof.
\end{proof}
\begin{remark}\label{eta.subopt}
Since the linearized error $\bxi$ (that is, $\bxi_1,\bxi_2$) is optimal in nature,
so we obtain from (\ref{ee2.1})-(\ref{ee2.2})
\begin{align*}
\|\bta_2\|_1^2 \le K(t)H^6+c\|\bta_1\|_1^2+cte^{-2\alpha t}\int_{t_0}^t e^{2\alpha s}\|\bta_1(s)\|_1^2 ds \\
\|\bta_2\|^2 \le K(t)H^8+cH^2\|\bta_1\|_1^2+cH^2te^{-2\alpha t}\int_{t_0}^t
e^{2\alpha s}\|\bta_1(s)\|_1^2 ds.
\end{align*}
Use Lemma \ref{l2.err1} to find that
\begin{eqnarray}
\|\bta_2\|_1^2 \le K(t)H^6+c\|\bta_1\|_1^2 \label{eta.H1} \\
\|\bta_2\|^2 \le K(t)H^8+cH^2\|\bta_1\|_1^2. \label{eta.L2}
\end{eqnarray}
\end{remark}
\noindent Following \cite{MX95}, we introduce the operator $R_h^H:\bJ_h\to \bJ_h^H$
satisfying
\begin{equation}\label{prop.R}
a(\bv-R_h^H\bv,\bchi)=0,~\forall\chi\in\bJ_h^H.
\end{equation}
With the notations
$$ ||\bv||_R=\|(I-R_h^H)\bv\|_1,~~~ (\bv,\bw)_R=a((I-R_h^H)\bv,(I-R_h^H)\bw), $$
we have, from Lemma 4.1 of \cite{MX95},
\begin{equation}\label{norm.R}
c_1\|\bv\|_1 \le \|\bv\|_R \le c_2\|\bv\|_1,
\end{equation}
where $c_1,c_2$ are positive constants independent of $h,H$.
And similar to Lemmas 4.5 and 4.6 of \cite{MX95}, we find for $\bphi\in\bJ_H$
\begin{eqnarray}
(\by_t^H,\bphi)+\mu (\by^H,\bphi)_R &=& (\f,(I-R_h^H)\bphi)-\int_{t_0}^t \beta(t-s)
a(\bu^h,(I-R_h^H)\bphi)~ds \nonumber \\
&&-b(\bu^h,\bu^h,(I-R_h^H)\bphi)-b(\bz^h,\bz^h,R_h^H\bphi) \label{eqn.R} \\
(\by_{H,t},\bphi)+\mu (\by_H,\bphi)_R &=& (\f,(I-R_h^H)\bphi)-\int_{t_0}^t \beta(t-s)
a(\bu_h,(I-R_h^H)\bphi)~ds \nonumber \\
&&-b(\bu_h,\bu_h,(I-R_h^H)\bphi)+(\bu_{h,t},R_h^H\bphi).
\end{eqnarray}
Now, for $\bphi\in\bJ_H$, we write the equation in $\e_1=\by_H-\by^H$ as
\begin{align}\label{eqn.e1.R}
(\e_{1,t},\bphi)+&\mu (\e_1,\bphi)_R= -\int_{t_0}^t \beta(t-s) a(\e_1,(I-R_h^H)\bphi)~ds
+(\bu_{h,t},R_h^H\bphi) \nonumber \\
+& b(\bz^h,\bz^h,R_h^H\bphi)-b(\e_1+\e_2,\bu_h,(I-R_h^H)\bphi)-b(\bu_h,\e_1+\e_2, (I-R_h^H)\bphi) \nonumber \\
+& b(\e_1+\e_2,\e_1+\e_2,(I-R_h^H)\bphi).
\end{align}
\begin{lemma}\label{err.e1.h1}
Under the assumptions Lemma \ref{err.phi}, we have
$$ \|\e_1\|_1^2+\int_{t_0}^t \|\e_{1,t}\|^2ds \le K(t)H^6. $$
\end{lemma}

\begin{proof}
Put $\bphi=\e_{1,t}$ in (\ref{eqn.e1.R}) and observe that
\begin{align*}
(\bu_{h,t},R_h^H\e_{1,t}) &=\frac{d}{dt}(\bu_{h,t},R_h^H\e_1)-(\bu_{h,tt},R_h^H\e_1) \\
&=\frac{d}{dt}(\bu_{h,t},R_h^H\e_1)-((I-P_H)\bu_{tt},R_h^H\e_1)-((\bu_h-\bu)_{tt},
R_h^H\e_1) \\
& \le \frac{d}{dt}(\bu_{h,t},R_h^H\e_1)+K(t)H^3\|\e_1\|_1, \\
-\int_{t_0}^t \beta(t-s)& a(\e_1(s),(I-R_h^H)\e_{1,t})~ds = \frac{d}{dt} \Big\{
-\int_{t_0}^t \beta(t-s) a(\e_1(s),(I-R_h^H)\e_1)~ds\Big\} \\
&+\beta(0) a(\e_1,(I-R_h^H)\e_1)-\delta\int_{t_0}^t \beta(t-s) a(\e_1(s),
(I-R_h^H)\e_1)~ds, \\
-b(\e_1+\e_2,\bu_h,& (I-R_h^H)\e_{1,t})-b(\bu_h,\e_1+\e_2,(I-R_h^H)\e_{1,t}) \\
& \le c(\|\e_1\|_1+\|\e_2\|_1)\|\bu_h\|_2\|\e_{1,t}\|, \\
b(\e_1+\e_2,\e_1+&\e_2,(I-R_h^H)\e_{1,t})=b(\e_1+\e_2,\bu_h-\bu^h,(I-R_h^H)\e_{1,t}) \\
& \le c(\|\e_1\|_1+\|\e_2\|_1)(\|\bu_h\|_2+\|\bu^h\|_2)\|\e_{1,t}\| \\
b(\bz^h,\bz^h,R_h^H\e_{1,t}) &=b(\bz_h-\e_2,\bz_h-\e_2,R_h^H\e_{1,t})
= \frac{d}{dt} b(\bz_h,\bz_h,R_h^H\e_1)-b(\bz_{h,t},\bz_h,R_h^H\e_1) \\
& -b(\bz_h,\bz_{h,t},R_h^H\e_1)+b(\bz_h-\e_2,-\e_2,R_h^H\e_{1,t})+b(\e_2,\bz_h,
R_h^H\e_{1,t}) \\
-b(\bz_{h,t},\bz_h,R_h^H\e_1)&-b(\bz_h,\bz_{h,t},R_h^H\e_1) \le cH\|\bz_h\|_1 \|\bz_{h,t}\|_1\|\e_1\|_1 \\
b(\e_2,\bz_h,R_h^H\e_{1,t})& \le c\|\e_2\|^{1/2}\|\e_2\|_1^{1/2}(\|\bz_h\|_1
\|R_h^H\e_{1,t}\|^{1/2}\|R_h^H\e_{1,t}\|_1^{1/2}+\|\bz_h\|^{1/2}\|\bz_h\|_1^{1/2}
\|R_h^H\e_{1,t}\|_1) \\
&\le cH\|\e_2\|_1\|\bz_h\|_1\|\e_{1,t}\|_1 \le c\|\e_2\|_1\|\bz_h\|_1\|\e_{1,t}\|.
\end{align*}
Here, we have used that $\|(I-R_h^H)\e_{1,t}\| \le \|\e_{1,t}\|+cH\|\e_{1,t}\|_1
\le c\|\e_{1,t}\|$. And now we have
\begin{align}\label{e1h101}
\|\e_{1,t}\|^2+\frac{\mu}{2}\frac{d}{dt}\|\e_1\|_R^2 \le K(t)H^3\|\e_1\|_1+
+c\|\e_1\|_1(\|\e_1\|_1+\int_{t_0}^t \beta(t-s)\|\e_1(s)\|_1) \nonumber \\
\frac{d}{dt} \Big\{(\bu_{h,t},R_h^H\e_1)-\int_{t_0}^t \beta(t-s) a(\e_1(s),
(I-R_h^H)\e_1)~ds+b(\bz_h,\bz_h,R_h^H\e_1)\Big\} \nonumber \\
+c(\|\e_1\|_1+\|\e_2\|_1) \|\e_{1,t}\|+cH\|\bz_h\|_1 \|\bz_{h,t}\|_1\|\e_1\|_1
+c\|\e_2\|_1\|\bz_h\|_1\|\e_{1,t}\|.
\end{align}
Integrate (\ref{e1h101}),use (\ref{norm.R}) and the fact that $\e_1(t_0)=0$ to find
\begin{align*}
\|\e_1\|_1^2+\int_{t_0}^t \|\e_{1,t}\|^2ds \le K(t)H^6+c\int_{t_0}^t (\|\e_1\|_1^2
+\|\e_2\|^2)~ds+(\bu_{h,t},R_h^H\e_1) \\
-\int_{t_0}^t \beta(t-s) a(\e_1(s),(I-R_h^H)\e_1)~ds+b(\bz_h,\bz_h,R_h^H\e_1).
\end{align*}
As earlier, we estimate the last three terms to obtain
\begin{align}\label{e1h102}
\|\e_1\|_1^2+\int_{t_0}^t \|\e_{1,t}\|^2ds \le K(t)H^6+c\int_{t_0}^t (\|\e_1\|_1^2
+\|\e_2\|^2)~ds.
\end{align}
We note from (\ref{eta.H1}) and triangle inequality that
$$ \|\e_2\|_1^2 \le K(t)H^6+\|\bta_1\|_1^2 \le K(t)H^6+\|\e_1\|_1^2. $$
Therefore
\begin{align*} 
\|\e_1\|_1^2+\int_{t_0}^t \|\e_{1,t}\|^2ds \le K(t)H^6+c\int_{t_0}^t \|\e_1\|_1^2ds.
\end{align*}
Use Gronwall's lemma to complete the rest of the proof
\end{proof}
\begin{remark}
The Lemma \ref{err.e1.h1} tells us that
$$ \|\bta_1\|_1 \le K(t)H^3, $$
and as a result, from Remark \ref{eta.subopt}, we have
$$ \|\bta_2\|+H\|\bta_2\|_1 \le K(t)H^4. $$
Another application of triangle inequality results in
$$ \|\e_2\|+H\|\e_2\|_1 \le K(t)H^4. $$
\end{remark}
\noindent
For the final estimate, we write down the error equations in terms of $\e_i,~i=1,2$.
\begin{align}\label{eqn.e12}
\left\{\begin{array}{rl}
(\e_{1,t},\bphi)+\mu a (\e,\bphi)+ \int_{t_0}^t \beta(t-s) a(\e(s), \bphi)~ds
&= -b(\bu_h,\bu_h,\bphi)+b(\bu^h,\bu^h,\bphi) \\
\mu a (\e,\bchi)+ \int_{t_0}^t \beta(t-s) a(\e(s), \bchi)~ds &= -(\bz_{ht},\bchi)
-b(\bu_h,\bu_h,\bchi) \\
& ~ +b(\bu^h,\bu^h,\bchi)-b(\bz^h,\bz^h,\bchi),
\end{array}\right.
\end{align}

\begin{lemma}\label{err.e1}
Under the assumptions Lemma \ref{err.e1.h1}, we have
$$ \|(\bu_h-\bu^h)(t)\| \le K(t)H^4, ~~t>t_0. $$
\end{lemma}

\begin{proof}
With the notation $\td_H=P_H(-\Delta_h)$, we choose $\phi=\td_H^{-1}\e_{1,t}$ in the
first equation of (\ref{eqn.e12}) to find
\begin{align}\label{ee101}
\|\e_{1,t}\|_{-1}^2+\frac{\mu}{2}\frac{d}{dt}\|\e_1\|^2+\int_{t_0}^t \beta(t-s)
(\e_1(s),\e_{1,t})~ds=& b(\e_1+\e_2,\bu_h,\td_H^{-1}\e_{1,t}) \nonumber \\
&+b(\bu^h,\e_1+\e_2,\td_H^{-1}\e_{1,t}).
\end{align}
Observe that
\begin{align*}
\frac{d}{dt}\Big\{\int_{t_0}^t \beta(t-s)(\e_1(s),\e_1)~ds \Big\}
=& \int_{t_0}^t \beta(t-s)(\e_1(s),\e_{1,t})~ds+\beta(0)\|\e_1\|^2 \\
&-\delta\int_{t_0}^t \beta(t-s)(\e_1(s),\e_1)~ds.
\end{align*}
Hence, we obtain from (\ref{ee101})
\begin{align}\label{ee102}
\|\e_{1,t}\|_{-1}^2+&\frac{\mu}{2}\frac{d}{dt}\|\e_1\|^2+
\frac{d}{dt}\Big\{\int_{t_0}^t \beta(t-s)(\e_1(s),\e_1)~ds \Big\}
+\delta\int_{t_0}^t \beta(t-s)(\e_1(s),\e_1)~ds \nonumber \\
&= \beta(0)\|\e_1\|^2+b(\e_1+\e_2,\bu_h,\td_H^{-1}\e_{1,t})
+b(\bu^h,\e_1+\e_2,\td_H^{-1}\e_{1,t}).
\end{align}
As earlier, we have
\begin{align*}
& b(\e_1+\e_2,\bu_h,\td_H^{-1}\e_{1,t})+b(\bu^h,\e_1+\e_2,\td_H^{-1}\e_{1,t}) \\
\le & c(\|\e_1\|+\|\e_2\|)(\|\bu_h\|_2+\|\bu^h\|_2)\|\e_{1,t}\|_{-1}.
\end{align*}
Integrate (\ref{ee102}) and use the above estimate to find
\begin{align*}
\|\e_1\|^2+\int_{t_0}^t \|\e_{1,t}\|_{-1}^2 \le c\int_{t_0}^t (\|\e_1\|^2+\|\e_2\|^2)~ds
\le K(t)H^8+c\int_{t_0}^t \|\e_1\|^2ds.
\end{align*}
Apply Gronwall's lemma to conclude
$$ \|\e_1\|^2+\int_{t_0}^t \|\e_{1,t}\|_{-1}^2 \le K(t)H^8. $$
Combining with the Remark 5.3, we have the desired result.
\end{proof}

\begin{remark}
It is clear from our above analysis is that the linearized error between NLG approximation and Galerkin approximation is of order $H^4$ in $L^2$-norm. However, non-linearized part of the error may not always be of same order. For example, if the equation in $\bz^h$ contains only $b(\by^H,\by^H,\bchi)$, then the non-linearized part of the error (i.e. the equation in $\bta$) will contain additional terms like $b(\by^H,\bz^h,\bchi)$ and $b(\bz^h,\by^H,\bchi)$
apart from the non-linear terms of the second equation of (\ref{bta12a}). And with one of these terms, we believe, we can only manage $H^3$ order of convergence in $L^2$-norm.
\end{remark}

\noindent
{\bf Acknowledge}: The author would like to thank CAPES for financial grant.


\begin{thebibliography}{25}

\bibitem{AM94}
Ammi, A and Marion, M ,
{\it Nonlinear Galerkin Methods and Mixed Finite Elements: Two-grid Algorithms for the
Navier-Stokes Equations},
Numer. Math. 57 (1994), 189-214.

\bibitem{BM97}
Burie, J and Marion, M ,
{\it Multilevel Methods in Space and Time for the Navier-Stokes Equations},
SIAM J. Numer. Anal. 34 (1997), 1574-1599.

\bibitem{CEHL99}
Cannon, J , Ewing, R , He, Y and Lin, Y ,
{\it A modified nonlinear {G}alerkin method for the viscoelastic fluid motion equations},
Internat. J. Engrg. Sci. 37 (1999), 1643-1662.

\bibitem{G11}
Goswami, D., {\it Finite Element Approximation to the Equations of Motion Arising in Oldroyd Viscoelastic Model of Order One},
Ph.D. Dissertation, Department of Mathematics, IIT Bombay (2011).

\bibitem{GP11}
Goswami, D and Pani, A,
{\it A Priori Error Estimates for Semidiscrete Finite Element Approximations to the 
Equations of Motion Arising in Oldroyd Fluids of Order One},
Int. J. Numer. Anal. Model. 8 (2011), 324-352.

\bibitem{GP08}
Guermond, J and Prudhomme, S ,
{\it A fully discrete nonlinear {G}alerkin method for the 3{D} {N}avier-{S}tokes equations},
Numer. Methods Partial Differential Equations 24 (2008), 759-775.

\bibitem{HL96}
He, Y and Li, K ,
{\it Nonlinear Galerkin Method and Two-step Method for the Navier-Stokes Equations},
Numer. Methods Partial Differential Equations 12 (1996), 283-305.

\bibitem{HL98}
He, Y and Li, K ,
{\it Convergence and Stability of Finite Element Nonlinear Galerkin Method for the Navier-Stokes Equations},
Numer. Math. 79 (1998), 77-106.

\bibitem{HL99}
He, Y and Li, K ,
{\it Optimum finite element nonlinear {G}alerkin algorithm for the
Navier-{S}tokes equations},
Math. Numer. Sin. 21 (1999), 29-38.

\bibitem{HLL99}
He, Y , Li, D and Li, K ,
{\it Nonlinear {G}alerkin method and {C}rank-{N}icolson method for viscous incompressible flow},
J. Comput. Math. 17 (1999), 139-158.

\bibitem{HR93}
Heywood, J and Rannacher, R ,
{\it On the question of turbulence modeling by approximate inertial manifolds and the
nonlinear {G}alerkin method},
SIAM J. Numer. Anal. 30 (1993), 1603-1621.

\bibitem{HMMC04}
He, Y , Miao, H , Mattheij, R , and Chen, Z ,
{\it Numerical analysis of a modified finite element nonlinear {G}alerkin method},
Numer. Math. 97 (2004), 725-756.

\bibitem{HWCL03}
He, Y , Wang, A , Chen, Z and Li, K ,
{\it An optimal nonlinear {G}alerkin method with mixed finite elements for the steady {N}avier-{S}tokes equations},
Numer. Methods Partial Differential Equations 19 (2003), 762-775.

\bibitem{LH10}
Liu, Q and Hou, Y ,
{\it A two-level finite element method for the {N}avier-{S}tokes equations based on a new projection},
Appl. Math. Model. 34 (2010), 383-399.

\bibitem{MT89}
Marion, M and Temam, R ,
{\it Nonlinear Galerkin Methods},
SIAM J Numer. Anal. 26 (1989), 1137-1159.

\bibitem{MT90}
Marion, M and Temam, R ,
{\it Nonlinear Galerkin Methods: The Finite Element Case},
Numer. Math. 57 (1990), 205-226.

\bibitem{MX95}
Marion, M and Xu, J ,
{\it Error Estimates on a New Nonlinear Galerkin Method Based on Two-grid Finite Elements},
SIAM J Numer. Anal. 32 (1995), 1170-1184.

\bibitem{NR97}
Nabh, G and Rannacher, R ,
{\it A Comparative Study of Nonlinear Galerkin Finite Element Methods for One-dimensional Dissipative Evolution Problems},
East-West J. Numer. Math. 5 (1997), 113-144.

\bibitem{temam}
Temam, R. , 
{\it  Navier-Stokes Equations Theory and Numerical Analysis}, 
North-Holland Publishing Co., Amsterdam, xii+526, 1984.
\end{thebibliography}
\end{document}